\documentclass[12pt]{article}
\usepackage[utf8]{inputenc}
\usepackage[english]{babel}
\usepackage[a4paper, total={6in, 10in}]{geometry}
\usepackage{graphicx} 
\usepackage{soul}
\usepackage{hyperref}
\usepackage[
style=numeric,
sorting=nyt,
backend=biber,
maxbibnames=7
]{biblatex}
\usepackage{amsmath, amsthm, amssymb, amsfonts, mathtools, mathrsfs}
\usepackage{bbm}
\usepackage{xfrac}
\usepackage{dsfont}
\usepackage{appendix}

\newtheorem{mainthm}{Theorem}

\usepackage{enumitem}

\newcommand{\lora}{\longrightarrow}

\newtheorem{theorem}{Theorem}[section]
\newtheorem*{thm}{Theorem}
\newtheorem{lemma}[theorem]{Lemma}
\newtheorem{proposition}[theorem]{Proposition}
\newtheorem{corollary}[theorem]{Corollary}
\newtheorem{definition}[theorem]{Definition}
\theoremstyle{definition}
\newtheorem{remark}[theorem]{Remark}
\newtheorem{example}[theorem]{Example}

\newtheorem{question}[theorem]{Question}

\newcommand\C{\mathbb C}
\newcommand\R{\mathbb R}
\newcommand\N{\mathbb N}
\newcommand\Z{\mathbb Z}
\newcommand\E{\mathbb E}
\renewcommand\P{\mathbb P}

\newcommand{\ler}[1]{\left( #1 \right)} 
\newcommand{\lers}[1]{\left\{ #1 \right\}} 
\newcommand{\lesq}[1]{\left[ #1 \right]}
\newcommand{\abs}[1]{\left| #1 \right|} 
\newcommand{\norm}[1]{\left\lVert#1\right\rVert}

\title{On the inclusion of bounded harmonic functions\\of random walks}

\author{
  Yair Hartman
  \thanks{Ben-Gurion University of the Negev; e-mail address: hartmany@bgu.ac.il}%
  \and
  Aranka Hrušková
  \thanks{Weizmann Institute of Science; e-mail address: umim.cist@gmail.com}%
  \and
Omer Segev
  \thanks{Technion – Israel Institute of Technology; e-mail address: omer.segev@campus.technion.ac.il}%
}

\date{January 2026}

\begin{document}

\maketitle

\begin{abstract}
    We investigate the conditions under which the space of bounded harmonic functions of a probability measure $\mu$ on a group $G$ is contained in that of another measure $\theta$. We establish that asymptotic commutativity, defined by the condition $\|\mu^{*t}*\theta - \theta*\mu^{*t}\|_{TV} \to 0$ as $t \to \infty$, is sufficient to guarantee the inclusion $H^\infty(G, \mu) \subseteq H^\infty(G, \theta)$, provided $\theta$ is absolutely continuous with respect to a convex combination of convolution powers of $\mu$. By employing martingale convergence techniques rather than ergodic-theoretic arguments, we demonstrate that this result holds without topological assumptions on $G$ (such as local compactness) and extends to general Markov chains. Furthermore, utilizing hitting models for the Poisson boundary, we characterise the inclusion $H^\infty(G, \mu) \subseteq H^\infty(G, \theta)$ as equivalent to the asymptotic invariance of $\theta$ under $\mu$ in the weak* topology. We apply these results to provide a probabilistic proof of the Choquet-Deny theorem for nilpotent groups, among other applications.
\end{abstract}

\section{Introduction}
Let $\mu$ be a probability measure on a group $G$. An integrable function $h\colon G\rightarrow \mathbb{C}$ is called \emph{$\mu$-harmonic} if $h(g)=\int_G h(g\gamma) \, d\mu(\gamma)$  for all $g\in G$. Denote by ${H}^\infty(G,\mu)$  the collection of all bounded $\mu$-harmonic functions on $G$.

In the study of bounded harmonic functions on groups, there are, generally speaking, two schools of thought. On one hand, there is the \textit{probabilistic approach}, which studies individual trajectories of random walks and their asymptotic behaviour, whose stabilisation is guaranteed by the martingale convergence theorem. On the other hand, we have the \textit{dynamical, global approach}, rooted in ergodic theory, where the primary object of study is the underlying group. In this perspective, one is concerned with stationary actions, operator algebras, and understanding the Poisson boundary (sometimes referred to as the \emph{Furstenberg-Poisson boundary}). The latter is a probability $G$-space $(X,\nu)$ which provides a presentation of bounded harmonic functions via an isomorphism between $L^\infty(X,\nu)$ and ${H}^\infty(G,\mu)$.

This latter field, established by the work of Furstenberg~\cite{furstenberg1963noncommuting,Furst71,furstenberg1973boundary}, Kaimanovich and Vershik~\cite{KaimVershik}, and many others, uses bounded harmonic functions as a tool for studying groups. For example, the lack of non-constant bounded harmonic functions for all generating measures $\mu$ is equivalent to the underlying countable group being FC-hypercentral
(virtually nilpotent, for fin. gen. groups~\cite{Annals2019}). The same condition for \textit{some} generating $\mu$ is equivalent to the group being
amenable (Rosenblatt~\cite{rosenblatt1981ergodic}, Kaimanovich-Vershik~\cite{KaimVershik}). As a corollary of our Theorem~\ref{thm:mainA}, we add another proof to the known list of proofs stating that nilpotent groups admit only constant bounded harmonic functions. Interestingly, even this result can be approached via group actions (Margulis~\cite{margulis1966positive}, Theorem 3.3 in Furstenberg's~\cite{furstenberg1998stiffness}), or via probabilistic study (Blackwell~\cite{blackwell1955transient}, Choquet-Deny~\cite{cd1960}, and our new proof below).

Each of the two perspectives brings its own tools and characteristics, and their combination yields the rich and deep theory of bounded harmonic functions on groups. In this paper, we employ both approaches to study the basic question: when do two random-walk measures on a group share the same collection of bounded harmonic functions, or, more generally, when is one a subset of the other? 

In this paper, we provide two answers to this question.
We mention that independently, Frisch and Chawla use inclusions of sets of bounded harmonic functions in their paper~\cite{2025nonrealizability}, showing that not all the Poisson boundaries of random walks on $\mathbb{F}_2$ are realised on the Gromov boundary.

Our investigation is motivated by a recurring heuristic in the theory: commutativity often acts as a rigidity mechanism for harmonic structures. For instance, the resolution of the Choquet-Deny problem identifies commutativity (modulo finite index and induction) as the only obstruction to the existence of non-trivial bounded harmonic functions~\cite{Annals2019}. Similarly, the commuting structure of the factors in the case of a product group forces the Poisson boundary to also decompose as a product~\cite{BaderShalom2006}. Another manifestation is found in the procedure of ``spreading out" a measure: taking a convex combination of convolution powers yields a new measure that commutes with the original one and leaves the space of bounded harmonic functions intact. Finally, this rigidity appears in the Poisson boundary itself, which admits no non-trivial $G$-equivariant endomorphisms. That is, there are no non-trivial maps that commute with the action~\cite{furstenberg2009stationary}. This statement holds whether one considers measurable point maps or unital positive maps on the algebra of bounded harmonic functions~\cite{hartman2024tight}. These examples suggest that measures commuting with that of a random walk tend to respect its harmonic functions, prompting us to look for a condition based on commutativity. Indeed, in Section~\ref{sec:applications}, we recover some of these known results as corollaries of Theorem~\ref{thm:mainA}.

Our main result is a sufficient condition for the inclusion of bounded harmonic functions in terms of \textit{asymptotic commutativity}.

\begin{mainthm}\label{thm:mainA}(Theorem~\ref{thm:asmpt-commuting}) 
    Let $G$ be a group equipped with a $G$-invariant $\sigma$-algebra $\Sigma$ and $\mu$, $\theta$ be two probability measures on $(G,\Sigma)$ such that $\theta\ll\sum_{n\geq0}\frac{1}{2^{n+1}}\mu^{*n}$ and
    \[
    \norm{\theta*\mu^{*t}-\mu^{*t}*\theta}_{TV}\to0 \text{ as } t\to\infty,
    \]
    where $*$ stands for convolution and norm is the total-variation norm. 

    Then ${H}^\infty (G,\mu)\subseteq {H}^\infty (G,\theta)$.
\end{mainthm}

    In particular, if the measures actually commute ($\mu*\theta=\theta*\mu$) and are both \textit{generating}, then they share exactly the same bounded harmonic functions. In the case of a locally compact second countable group $G$, and Borel admissible probability measures $\mu$ and $\theta$ on $G$, it is shown in~\cite{BaderShalom2006} and~\cite{BjHaOppel2023} that if $\mu*\theta=\theta*\mu$ then every $(G,\mu)$-stationary action is also $(G,\theta)$-stationary, and vice versa. It can be shown using the Poisson boundary theory that in this case this statement is equivalent to the statement that $H^\infty(G,\mu)=H^\infty(G,\theta)$. However,    
    our proof is based on a careful treatment of martingales and not on the approach of group actions. This allows us to consider non-generating measures $\theta$, and using them, to conclude the aforementioned Choquet-Deny theorem for nilpotent groups.
    The main technical lemma states that asymptotically, the values of a bounded harmonic function observed by a random walk do not change if one adds an additional, constant step \textit{at the end} of the walk (Proposition~\ref{prop:theBasicLemma}).
    
    Thanks to our proof method, we obtain the result for the generality of asymptotic commutativity and with no restrictions on the group; it does not require local compactness or second countability. Furthermore, as often happens with martingale arguments, when treated carefully, the result actually holds for general Markov chains, after making the correct adaptations to the statement. We discuss these and the more general proof in Subsection~\ref{subsection:MarkovCh}.

While asymptotic commutativity is a sufficient condition for inclusion, it is not a necessary one -- see Example~\ref{example}. The correct if-and-only-if condition is ``$\mu$ being asymptotically $\theta$-invariant''. But in what sense?
We bring here a classical result of Derriennic, Theorem~\ref{thm:Derriennic}, which shows that this asymptotic invariance can be taken in the total-variation norm -- \textit{(2)} in Theorem~\ref{thm:mainB}. We translate Derriennic's proof, which holds for general Markov chains, from French in the appendix, and add a dynamical characterisation, in which the asymptotic invariance can be taken in the weak* sense -- \textit{(3)} in Theorem~\ref{thm:mainB}. For that, one should be careful about the underlying topology. We show in Lemma~\ref{lemma:hitting}  that one can use weak* topology whenever  Poisson boundary is realized as a ``hitting measure'' of the random walk. Furthermore, we show in Section~\ref{Sec:hitting} that such models exist for every admissible measure on a locally compact second countable group.
This last result, showing that $\mu$-boundaries can be realised as hitting measures, might be of independent interest for stationary theory.

This discussion is summarised by the following theorem.

\begin{mainthm}\label{thm:mainB}
(Theorem~\ref{thm:TFAE}) 
    Let $G$ be a locally compact second countable group and let $\mu$ be an admissible Borel probability measure on $G$. Let $\theta$ be a Borel probability measure on $G$.
    Consider a hitting model $G \cup K$ for the Poisson boundary of $(G,\mu)$.

    Then the following are equivalent.

    \begin{enumerate}[label=(\arabic*)]
        \item Every bounded $\mu$-harmonic function is $\theta$-harmonic: $H^{\infty}(G,\mu)\subseteq H^\infty(G,\theta).$
        \item  $\|\theta*\tilde{\mu}^{*t}-\tilde{\mu}^{*t}\|_{TV} \to 0$ where TV is the total-variation norm and $\tilde{\mu}=\frac{1}{2}(\mu+\delta_e)$.
        \item  $\theta*\mu^{*t}-\mu^{*t}\to 0$ in the weak* topology on $Prob(G\cup K)$. 
    \end{enumerate}
\end{mainthm}

This is another manifestation of the differences between the two approaches to studying bounded harmonic functions. The proof of the ergodic characterisation is softer but is restricted only to random walks on groups, whereas the probabilistic one applies to general Markov chains. Note that in Theorem~\ref{thm:mainA} we allow an arbitrary $G$-invariant $\sigma$-algebra, whereas in Theorem~\ref{thm:mainB} we restrict to the Borel $\sigma$-algebra, since we address ergodic considerations.
But even if it was not for this greater generality of the setup, Theorem~\ref{thm:mainA} does not follow at once from Theorem~\ref{thm:mainB}, and we are as of now only able to show that $\|\theta*\tilde{\mu}^{*t}-\tilde{\mu}^{*t}\|_{TV} \to 0$ implies $\norm{\theta*\mu^{*t}-\mu^{*t}*\theta}_{TV}\to0$ by combining the two main theorems.

We remark that it is essential to consider \emph{bounded} harmonic functions throughout the discussion. While the martingale convergence theorem holds for positive functions too, even abelian groups - on which any two measures commute - admit random walks with non-constant positive harmonic functions, and generally without inclusions between the corresponding spaces of positive harmonic functions, see for example~\cite{SawyerMartinBoundary}.

\addvspace{20pt}

\textbf{Acknowledgements.}
We would like to thank Tomer Zimhoni for many helpful discussions. YH also thanks Katok Center for Dynamical Systems and Geometry of Pennsylvania State University for their hospitality.
This research was supported by the European Research Council (ERC) under the European Union's Horizon Europe research and innovation programme (Grant agreement No. 101078193).

\section{Background}\label{background}

Let us imagine that there exists an object, a sort of boundary, which almost every $\mu$-random walk on $G$ eventually `hits'. If this object happened to be a measurable space, we could consider an event $E$ and for every element $g\in G$ ask what the probability $\P_g(E)$ of hitting the boundary at $E$ after starting the random walk from $g$ is. Then $\P_g(E)$ would be a (bounded) harmonic function simply by the law of total probability. More generally, we could take a bounded function $F$ on the boundary and ask for the expectation $\E_g[F]$ of the value of $F$ at which our random walk started from $g$ hits it.

The following is the foundational theorem of the boundary theory of bounded harmonic functions on groups, which proves
that such an object indeed exists and that remarkably \emph{all} bounded harmonic functions arise in the way described above.

\begin{theorem}[Furstenberg, \cite{Furst71}]\label{thm:Furstenberg}
    Let $G$ be a second countable, locally compact group and $\mu$ a Borel probability measure on $G$.
    Then there exists a $\mu$-stationary probability $G$-space $(B,\nu)$ such that
    \[
    H^\infty(G,\mu)\cong L^\infty(B,\nu),
    \]
    where the isomorphism is that of normed $G$-spaces.
\end{theorem}

Recall that a measured $G$-space $(X,\lambda)$ is \emph{$\mu$-stationary} if $\lambda=\int_G g\cdot\lambda \ d\mu(g)$. The probability space $(B,\nu)$ above is known as the \emph{Furstenberg-Poisson boundary} or just the \emph{Poisson boundary} of $(G,\mu)$. See also~\cite{KaimVershik} for more details.
Furthermore, $G$-factors of the Poisson boundary are called \textit{$\mu$-boundaries}. That is, for every shift-invariant, $G$-equivariant, measurable map $\Pi$ from the path space $(G^\mathbb{N},{P_{\mu}})$ to a $G$-space $X$, the pair $(X,\Pi_* P_\mu)$ is a $\mu$-boundary. When a $\mu$-stationary $G$-space $(X,\lambda)$ is also a topological space, it is a $\mu$-boundary exactly when for almost every trajectory $(w_n)$ of the $\mu$-random walk, 
\[
w_n\lambda \xrightarrow{w^*} \delta_x,
\] 
for some $x\in X$.  For further details, see for example~\cite{Kaim2000}.

The correspondence of $L^\infty(B,\nu)$ and $H^\infty(G,\mu)$ is closely related to the martingale convergence theorem, which we will also use extensively in our proofs.
A \emph{martingale} is a sequence of complex-valued random variables such that at each step, the conditional expectation of the next value, given the past, is equal to the current value. Harmonic functions naturally give rise to martingales: if $(X_t)_{t=0}^\infty$ is the $\mu$-random walk on $G$ and a function on $G$ $h$ is a $\mu$-harmonic, then for every $g\in G$, the sequence $(h(gX_t))_{t=0}^\infty$ is a martingale.
The martingale convergence theorem then guarantees that every bounded martingale converges almost surely.

In practice, many boundary results only hold or are known for some classes of sufficiently well behaved measures such as the following
\begin{definition}
    Let $G$ be a topological group and $\mu$ a Borel probability measure on $G$.
    \begin{enumerate}
        \item We say that $\mu$ is \emph{generating} if the closed semi-group generated by $supp(\mu)$ equals $G$, and that it is \emph{adapted} if the closed group generated by $supp(\mu)$ equals $G$.
        \item If $G$ is locally compact and second countable, $\mu$ is \emph{spread-out} if there exists $n$ such that $\mu^{*n}$ is non-singular with respect to the Haar measure on $G$.
        \item If $\mu$ is both generating and continuous with respect to the Haar measure, we call it \emph{admissible}.
        \item We say that $\mu$ is \emph{lazy} if $\mu(e)>0$. 
    \end{enumerate}
\end{definition}

As explained in Introduction, the method of our proof allows for a generalisation to Markov chains on state spaces with no algebraic structure.
We recall that a \emph{Markov kernel}, or \emph{transition kernel}, $P\colon M\times\mathcal{F}\lora[0,1]$ on a measurable space $(M,\mathcal{F})$ is a map such that for every fixed $x\in M$, the map $P(x,\cdot)$, also denoted $P_x$, is a probability measure on $(M,\mathcal{F})$ and such that for every $A\in\mathcal{F}$, the map $P(\cdot,A)$ on $M$ is $\mathcal{F}$-measurable.

Given some initial distribution $\kappa$, a Markov chain $\ler{X_n}_{n\geq0}$ is then defined by $X_0\sim\kappa$ and
\[
\mathbb{P}\ler{X_{n+1}\in A \ | \ X_n=x}=P_x(A).
\]

Just as for groups, the $P$-harmonic functions on $M$
\[
H(M,P)=\lers{h\colon M\rightarrow \R \ \bigg| \ h(x)=\int_{M}h(y)\, dP_x(y) \text{ for all } x\in M}
\]
are those integrable functions whose expected value after making a $P$-random step equals the value from which we departed.
In other words, $H(M,P)$ are the eigenfunctions with eigenvalue 1 of the Markov operator $P^*\colon f\mapsto\ler{x\mapsto\int f\, dP_x}$.

\section{Asymptotically commuting measures}\label{section:proof}

\begin{definition}
Let $\mu,\theta$ be two probability measures on $G$. If
\[
\norm{\mu^{*t}*\theta-\theta*\mu^{*t}}_{TV}\to0 \text{ as } t\to\infty, 
\]
we say that \emph{$\mu$  asymptotically commutes with $\theta$} or that \emph{$\mu$  asymptotically centralises $\theta$}.
\end{definition}

The following example demonstrates that the set of probability measures which asymptotically commute with a fixed measure $\theta$ can indeed be strictly larger than that of those which commute with $\theta$.

\begin{example}\label{example}
      Let $L(\Z)=\bigoplus_{\Z} 
      \sfrac{\Z}{2\Z} \rtimes \Z$ be the lamplighter group with $\theta=\delta_{(\mathds{1}_0,0)}$ representing the deterministic action of switching the lamp at the lamplighter's current position and $\mu=\frac{1}{4}\ler{\delta_{(0,1)}+\delta_{(\mathds{1}_1,1)}+\delta_{(0,-1)}+\delta_{(\mathds{1}_{-1},-1)}}$ being the `walk-switch' measure. Then the two convolutions $\theta*\mu$ and $\mu*\theta$ have disjoint supports, because $\theta*\mu$ always switches the lamp at the lamplighter's initial position while $\mu*\theta$ never does, but $\norm{\theta*\mu^{*t}-\mu^{*t}*\theta}_{TV}\to0$ as $t\to\infty$.
    This is because the $\mu$-random walk leaves the lamps of all the integers visited after time 0 uniformly on or off both with probability 1/2, so $\mu^{*t}$ and $\mu^{*t}*\theta$ are in fact the same distribution for every $t\geq1$.
    When the walker on $L(\Z)$ makes one $\theta$-random step and then proceeds with $\mu$-random steps, the configuration of the initial lamp becomes uniformised when the walker revisits it, which happens with probability 1, because the simple random walk on $\Z$ is recurrent.
\end{example}

Example~\ref{example} also shows that asymptotic commutativity is not a symmetric relation in general.
To get a pair of probability measures such that either of them asymptotically centralises the other while they do not commute, one can replace $\theta=\delta_{(\mathds{1}_0,0)}$ above with $\nu=\frac{1}{2}\theta+\frac{1}{2}\mu$. We, however, have no example at this point of a hyperbolic group and two probability measures on it such that each of them asymptotically centralises the other one, while they do not commute. In fact, our guess is that this cannot happen.

\begin{question}
    Let $G$ be a hyperbolic group and $\mu,\theta$ two Borel probability measures on $G$ such that both $\norm{\mu^{*t}*\theta-\theta*\mu^{*t}}_{TV}\to0 \text{ as } t\to\infty$ and $\norm{\theta^{*t}*\mu-\mu*\theta^{*t}}_{TV}\to0 \text{ as } t\to\infty $. Is it necessarily the case that $\mu*\theta=\theta*\mu$?
\end{question}

In the following section, we will prove that if for a function $f$ on $G$, the sequence $f\ler{w_n}$ converges almost surely, where $\ler{w_n}_{n=0}^\infty$ is a trajectory of the $\mu$-random walk, then the translates $f\ler{w_n\gamma}$ also converge almost surely. From this and the martingale convergence theorem, we will deduce Theorem~\ref{thm:mainA}, and then show how its statement would follow from a proposition of Bader and Shalom if we were only interested in restricting our attention to locally compact second countable groups and admissible commuting measures.

\subsection{Proof of Theorem~\ref{thm:mainA}}

Let $G$ be a group equipped with a $G$-invariant $\sigma$-algebra $\Sigma$ and $\mu$ a probability measure on $(G,\Sigma)$. We denote by $\nu_\mu$ the probability measure on $G^\N$ which is the law of the $\mu$-random walk $X_0=e, X_1, X_2, \dots$, that is the random trajectory whose increments $X_i^{-1}X_{i+1}$ are distributed according to $\mu$ and independently of one another. In other words, the convolution power $\mu^{*m}$ is the law of $X_m$.

\begin{proposition}\label{prop:theBasicLemma}
    Let $G$, $\mu$ and $\nu_\mu$ be as above. Let $f\colon G\lora\C$ be a measurable function such that for $\nu_\mu$-almost every trajectory $\ler{w_n}_n\in G^\N$, the values $f\ler{w_n}$ converge. Let $m$ be a fixed positive integer. Then for $\nu_\mu\times\mu^{*m}$-almost every pair $\ler{(w_n)_n,\gamma}\in G^\N\times G$, we have
    \begin{equation}\label{converges}
        |f\ler{w_n}-f\ler{w_n\gamma}|\to0 \text{ as } n\to\infty.
    \end{equation}
\end{proposition}

\begin{proof}
    The informal slogan of the proof is that if something has enough of a chance of happening, then it actually does happen.

    The set of trajectories and increments that do not satisfy (\ref{converges}),
    \begin{align*}
        B&:=\lers{\ler{(w_n),\gamma}\in G^\N\times G : |f\ler{w_n}-f\ler{w_n\gamma}|\not\to0}\\
        &=\lers{\ler{(w_n),\gamma}\in G^\N\times G : \exists\varepsilon>0 \text{ such that } |f\ler{w_n}-f\ler{w_n\gamma}|>\varepsilon \text{ for infinitely many } n},
    \end{align*}
    can be written as the increasing union
    \[
    \bigcup_{k=1}^\infty\lers{\ler{(w_n),\gamma}\in G^\N\times G : |f\ler{w_n}-f\ler{w_n\gamma}|>\frac{1}{k} \text{ for infinitely many } n}.
    \]
    Let us suppose on the contrary that
    \[
    \ler{\nu_\mu\times\mu^{*m}}(B)>0.
    \]
    Then there exists a positive integer $K$ such that
    \[
    \ler{\nu_\mu\times\mu^{*m}}\ler{\bigcup_{k=1}^K\lers{\ler{(w_n),\gamma}\in G^\N\times G : |f\ler{w_n}-f\ler{w_n\gamma}|>\frac{1}{k} \text{ for infinitely many } n}}>0.
    \]
    In other words,
    \[
    \ler{\nu_\mu\times\mu^{*m}}\ler{B_K}=\delta
    \]
    for some $\delta>0$, where $B_K\subseteq B$ is the set
    \[
    B_K:=\lers{\ler{(w_n),\gamma}\in G^\N\times G : |f\ler{w_n}-f\ler{w_n\gamma}|>\frac{1}{K} \text{ for infinitely many } n}.
    \]

    At the same time, $\nu_\mu$-almost every sequence $\ler{f(w_n)}_n$ is Cauchy, so there is a nullset $M\subset G^\N$ such that
    \begin{equation}\label{eq:Cauchy}
        G^\N\setminus M=\bigcup_{N=1}^\infty\lers{(w_n)_n\in G^\N : |f(w_i)-f(w_j)|<\frac{1}{K} \text{ for all } i,j\geq N}.
    \end{equation}
    Similarly as above, the left-hand side of (\ref{eq:Cauchy}) is an increasing countable union of nested sets, whose measure is 1, so there exists $T\in\N$ such that
    \[
    \nu_\mu\ler{\lers{(w_n)_n\in G^\N : |f(w_i)-f(w_j)|<\frac{1}{K} \text{ for all } i,j\geq T}}>1-\delta.
    \]
    That is to say, the proportion of walks for which $\frac{1}{K}$-Cauchyness kicks in by time $T$ is greater than $1-\delta$.

    Let us now partition $B_K$ into sets $R_0, R_1, R_2, \dots$ according to when after the $T$-th step a trajectory $(w_n)$ has the first opportunity to violate $\frac{1}{K}$-Cauchyness \emph{if} it steps by $\gamma$.
    Formally,
    \begin{align*}
        R_0&:=\lers{\ler{(w_n),\gamma}\in B_K : |f(w_T)-f(w_T\gamma)|>\frac{1}{K}}\\
        R_i&:=\lers{\ler{(w_n),\gamma}\in B_K : |f(w_{T+i})-f(w_{T+i}\gamma)|>\frac{1}{K}}\setminus {\bigcup_{j=0}^{i-1}R_j}, \text{ for }i\geq1.
    \end{align*}
    $R_i^\to$ will now be the set of (possibly) alternative futures in which the trajectory \emph{actually} steps by $\gamma$ at time $T+i$, causing a change by more than $\frac{1}{K}$ in the value of $f$.

    We have
    \begin{align*}
    R_i^\to:=\big\{\ler{x_n}_n\in G^\N : &\text{ there exists } \ler{(w_n),\gamma}\in R_i \\
    &\text{ such that } x_j=w_j \text{ for all }j\in[T+i],
    \text{ and } x_{T+i+m}=w_{T+i}\gamma\big\},
    \end{align*}
    meaning $\nu_\mu\ler{R_i^\to}\geq\ler{\nu_\mu\times\mu^{*m}}\ler{R_i}$.
    Then the sets $R_0^\to, R_1^\to, R_2^\to,\dots$ and
    \[
    C=\lers{(w_n)\in G^\N : |f(w_i)-f(w_j)|<\frac{1}{K} \text{ for all } i,j\geq T}
    \]
    are all pairwise disjoint, but also
    \[\sum_{i=0}^\infty\nu_\mu\ler{R_i^\to}\geq\sum_{i=0}^\infty\ler{\nu_\mu\times\mu^{*m}}\ler{R_i}=\ler{\nu_\mu\times\mu^{*m}}\ler{B_K}=\delta.
    \]
    This gives
    \[1=\nu_\mu\ler{G^\N}\geq\sum_{i=0}^\infty\nu_\mu\ler{R_i^\to}+\nu_\mu\ler{C}>\delta+(1-\delta)=1,
    \]
    which is a contradiction, so we must have $\ler{\nu_\mu\times\mu^{*m}}(B)=0$ as desired.  
\end{proof}

In a version for discrete groups only (where the condition $\theta\ll\sum_{n\geq0}\frac{1}{2^{n+1}}\mu^{*n}$ is equivalent to saying that $\text{supp} (\theta)\subseteq\bigcup_n\text{supp}\ler{\mu^{*n}}$), the proposition appears in the dissertation of the second author, and for $f$ a bounded real-valued harmonic function, a different proof can be found in Section 4 of Babillot's contribution in~\cite{proceedings}.

Equipped with it, we are now ready to prove Theorem~\ref{thm:mainA}, which we restate for the reader's convenience.

\begin{theorem}
    \label{thm:asmpt-commuting}
    Let $G$ be a group equipped with a $G$-invariant $\sigma$-algebra $\Sigma$ and $\mu$, $\theta$ two probability measures on $(G,\Sigma)$ such that $\theta\ll\sum_{n\geq0}\frac{1}{2^{n+1}}\mu^{*n}$ and
    \[
    \norm{\mu^{*t}*\theta-\theta*\mu^{*t}}_{TV}\to0 \text{ as } t\to\infty.
    \]
    Then $H^\infty(G,\mu)\subseteq H^\infty(G,\theta)$, that is, every bounded $\mu$-harmonic function is also $\theta$-harmonic.
\end{theorem}

The condition $\theta\ll\sum_{n\geq0}\frac{1}{2^{n+1}}\mu^{*n}$ means that all places that a $\theta$-random walk can visit are also potentially reachable by the $\mu$-random walk. Without this sort of dominance of $\mu$ over $\theta$, one cannot expect that information about $\mu$-harmonicity will imply $\theta$-harmonicity. Indeed, Example~\ref{Example:unrelated measures} shows that this condition is necessary for the theorem to hold.

\begin{proof}
    Let $h$ be a bounded $\mu$-harmonic function and $g$ a fixed element of $G$. Our aim now is to show that $h$ is $\theta$-harmonic at $g$. To that end, we define random variables
    \[
    M_t:=h\ler{gVX_t} \ \ \ \ \text{ and } \ \ \ \  N_t:=h\ler{gX_tV},
    \]
    where $X_t\sim\mu^{*t}$ is the composition of $t$ independent $\mu$-random jumps and $V\sim\theta$ is independent of $X_t$. We will show the following three claims.
    \begin{enumerate}[label=\arabic*)]
        \item $\abs{\E\lesq{M_t}-\E\lesq{N_t}}\to0$ as $t\to\infty$,
        \item $\E\lesq{N_t}\to h(g)$ as $t\to\infty$, and
        \item $\E\lesq{M_t}=\int h(gy) \, d\theta(y)$.
    \end{enumerate}

    Once these are established, we are done, because 1) says that $\E\lesq{M_t}$ and $\E\lesq{N_t}$ are increasingly similar as $t\to\infty$, while 2) and 3) tell us that $\E\lesq{N_t}$ is more and more like $h(g)$ and $\E\lesq{M_t}$ is even equal to $\int h(gy) \, d\theta(y)$. So we must have that $\abs{h(g)-\int h(gy) \, d\theta(y)}\to0$ as $t\to\infty$, but since the expression has in fact no dependency on $t$, we get that $h(g)=\int h(gy) \, d\theta(y)$, or in other words that $h$ is $\theta$-harmonic at $g$. But $g$ was arbitrary, so $h$ is $\theta$-harmonic on the entire group.

    Let us now prove the three claims. The last one is nothing but $\mu$-harmonicity of $h$ through
    \[
    \E\lesq{M_t}=\int_{G\times G} h(gyx) \, d\ler{\theta\times\mu^{*t}}(y,x) = \int_G\int_G h(gyx) \, d\mu^{*t}(x)d\theta(y) = \int_G h(y) \, d\theta (y),
    \]
    and the first one is shown by
    \begin{multline*}
     \abs{\E\lesq{M_t}-\E\lesq{N_t}}=\abs{\int_G h(gy) \, d\ler{\theta*\mu^{*t}-\mu^{*t}*\theta}(y)} \\
    \leq \norm{h}_\infty \norm{\theta*\mu^{*t}-\mu^{*t}*\theta}_{TV} \to0 \text{ as } t\to\infty.   
    \end{multline*}
    To prove the second claim, let us consider
    \begin{align}\label{ineq:2ndclaim}
    \abs{\E\lesq{N_t}-h(g)}&=\abs{\E\lesq{h(gX_tV)}-\E\lesq{h(gX_t)}} \nonumber \\
    &\leq\E\lesq{\abs{h(gX_tV)-h(gX_t)}}=\int_{G\times G} \abs{h(gxv)-h(gx)} \, d\ler{\mu^{*t}\times\theta}(x,v),
    \end{align}
    where we again used the $\mu$-harmonicity of $h$ to replace $h(g)$ with $\E\lesq{h(gX_t)}$. By the bounded martingale convergence theorem, $\ler{h(gw_n)}_{n=0}^\infty$ converges for almost every $\mu$-random trajectory $\ler{w_n}_{n=0}^\infty$, so we are in a position to employ Proposition~\ref{prop:theBasicLemma}, to get that
    \[
    \nu_\mu\times\Tilde{\mu}\ler{\lers{(w_n,\gamma)\in G^\N\times G : \abs{h(gw_n)-h(gw_n\gamma)}\not\to0}}=0.
    \]
    But $\theta$ is absolutely continuous with respect to $\Tilde{\mu}$, and so
    \[
    \nu_\mu\times\theta\ler{\lers{(w_n,\gamma)\in G^\N\times G : \abs{h(gw_n)-h(gw_n\gamma)}\not\to0}}=0
    \]
    as well.
    This implies that for every $\varepsilon>0$, there is $T\in\N$ such that
    \[
    \nu_\mu\times\theta\ler{\lers{(w_n,\gamma) : \abs{h(gw_n)-h(gw_n\gamma)}<\varepsilon \text{ for all }n\geq T}}>1-\varepsilon,
    \]
    and so in particular for any $t\geq T$,
    \begin{align*}
    \mu^{*t}&\times\theta\ler{\lers{(w_n,\gamma) : \abs{h(gw_n)-h(gw_n\gamma)}<\varepsilon}} \\
    \geq\nu_\mu&\times\theta\ler{\lers{(w_n,\gamma) : \abs{h(gw_n)-h(gw_n\gamma)}<\varepsilon \text{ for all }n\geq T}}
    >1-\varepsilon.
    \end{align*}
    Returning back to (\ref{ineq:2ndclaim}), we can now augment it to
    \begin{align}\label{ineq:dominated}
        \abs{\E\lesq{N_t}-h(g)}\leq\int_{G\times G}\abs{h(gxv)-h(gx)} \, d\ler{\mu^{*t}\times\theta}(x,v)
        < \varepsilon\cdot(1-\varepsilon)+2\norm{h}_\infty\cdot\varepsilon,
    \end{align}
    concluding that $\abs{\E\lesq{N_t}-h(g)}\xrightarrow{t\rightarrow\infty}0$, or in other words $\E\lesq{N_t}\xrightarrow{t\rightarrow\infty}h(g)$ as claimed.
\end{proof}

Since the martingale convergence theorem tells us that the assumption of Proposition~\ref{prop:theBasicLemma} of almost sure convergence is satisfied not just by bounded harmonic functions, but even positive harmonic functions, one could wonder whether Theorem~\ref{thm:asmpt-commuting} is extendable to them as well. Nonetheless, the answer is no -- it would fail already in the basic setup of abelian groups.

For example, for any positive integer $q$, the function $f(n)=q^{n}$, for $n\in\Z$, is $\mu_q$-harmonic for $\mu_q=\frac{q}{q+1}\delta_{-1}+\frac{1}{q+1}\delta_1$. Now, for $q_1 \ne q_2$ we will get no inclusion of positive harmonic functions, even though the measures $\mu_{q_1}$ and $\mu_{q_2}$ commute (as the group is abelian).

The failure of Theorem~\ref{thm:asmpt-commuting} for positive harmonic functions has a curious implication.
If measures $\mu$ and $\theta$ on $G$ commute, the only place at which our proof could (and hence, at least on occasion, must) fail is inequality (\ref{ineq:dominated}). What this means is the following.
Let $h$ be a positive $\mu$-harmonic function on $G$, and suppose that we start a number of independent $\mu$-random walks, say, from the identity. As time passes, for most of them, the values of $h$ start to be locally nearly constant, where by `locally', we understand the elements of the group to which the walker could transition within afew steps with high probability. Only for an increasingly more negligible fraction of the walks is this not so, but the extent to which it is not must be so large that the right-hand side of inequality~(\ref{ineq:dominated}) fails to go to 0. In other words, as time progresses and more and more $\mu$-random walks see the values of $h$ become increasingly alike in their vicinity, those fewer and fewer walks that do not must exhibit a very extreme opposite behaviour and see the nearby values of $h$ differ more and more from the value at which the walker is standing.

\subsection{General Markov chains}\label{subsection:MarkovCh}

Our proof can be generalised to the wider context of Markov chains on spaces with no algebraic structure.
Let $P\colon M\times\mathcal{F}\lora[0,1]$ be a Markov kernel on a measurable space $(M,\mathcal{F})$, and let us consider the set
\[
H^\infty(M,P)=\lers{h\colon M\rightarrow \R \ \bigg| \ h\in\ell^\infty(M,\mathcal{F}), h(x)=\int_{M}h(y)\, dP_x(y) \text{ for all } x\in M}
\]
of bounded $P$-harmonic functions on $M$.

\begin{theorem}\label{thm:generalMarkov}
    Let $P,Q$ be two Markov kernels on a measurable space $(M,\mathcal{F})$
    such that for every $x\in M$,
    \[
    \norm{\delta_x P^t Q-\delta_x Q P^t}_{TV}\to0 \ \text{ as } \ t\to\infty.
    \]
    Suppose further that there is $c>0$ such that for all $x\in M, A\in\mathcal{F}$,
    \[
    Q_x(A)\leq c\cdot \tilde{P}_x(A),
    \]
    where $\tilde{P}=\sum_{n\geq0}\frac{1}{2^{n+1}}P^n$.
    Then
    \[
    H^\infty(M,P)\subseteq H^\infty(M,Q),
    \]
    that is, every bounded $P$-harmonic function is also $Q$-harmonic.
\end{theorem}

For every $m\in M$, let $\ler{X_t^{(m)}}_{t=0}^\infty$ be the $P$-random walk started from $m$.
Moreover, let $\ler{Z_t^{(m)}}_{t=0}^\infty$ be obtained from $\ler{X_t^{(m)}}_{t=0}^\infty$
by $Z_t^{(m)}$ being the position of our walk after one independent $Q$-random step from $X_t^{(m)}$.

Lastly, let $\ler{W_t^{(m)}}_{t=0}^\infty$ be the $P$-random walk whose initial distribution is $Q_m=\delta_mQ$.

In terms of laws, for every $A\in\mathcal{F}$, we have
\begin{align*}
\P\ler{X_t^{(m)}\in A}&=\ler{\delta_mP^t}(A),\\
\P\ler{Z_t^{(m)}\in A}&=\ler{\delta_mP^tQ}(A),\\
\P\ler{W_t^{(m)}\in A}&=\ler{\delta_mQP^t}(A),
\end{align*}
and we denote the probability measure on $M^\N$ which is the law of the $P$-random walk $\ler{X_t^{(m)}}_{t}$ started at $m$ by $\nu_m$.

\begin{proposition}[Analogue of Proposition \ref{prop:theBasicLemma}]\label{prop:LemmaForMCh}
    Let a state $m\in M$ be fixed and suppose that there is $c>0$ such that $Q_x(A)\leq c\cdot \tilde{P}_x(A)$ for all $x\in M, A\in\mathcal{F}$.
    Let $f\colon M\lora\C$ be a measurable function such that for $\nu_m$-almost every trajectory $\ler{w_t}_t\in M^\N$, the values $f\ler{w_t}$ converge. Then
    \begin{align}\label{expr:convergence}
    f\ler{X_t^{(m)}}-f\ler{Z_t^{(m)}}\xrightarrow{t\rightarrow\infty}0 \ \ \ \ \nu_m \text{-almost surely.}     
    \end{align}
\end{proposition}
There is a subtlety in the statement ``$\nu_m$-almost surely'', which becomes clear once reading the beginning of the proof.
\begin{proof}
Let
 \[
    B:=\{(x_t)\in M^\N : \mathbb{P}\ler{f(x_t)-f(Y_t) \not\to0}>0 \text{ where }Y_t\sim Q_{x_t}\}
    \]
be the set of bad trajectories that do not satisfy (\ref{expr:convergence}).
The set $B$ can be written as the increasing union $\bigcup_{k\geq1} B_k$, where
\[
B_k:=\lers{(x_t)_t\in M^\N : \P\ler{|f(x_t)-f(Y_t)|>\frac{1}{k} \text{ for inf. many }t}>0\text{, where }Y_t\sim Q_{x_t}},
\]
    so if the proposition does not hold, meaning $\nu_m(B)>0$, then there exists a positive integer $K$ such that $\nu_m(B_K)>0$.
    Now since the events
    \[
    \lers{|f(x_t)-f(Y_t)|>\frac{1}{K}}
    \]
    are independent, we have that
    \begin{align*}
    B_K&=\lers{(x_t)_t\in M^\N : \mathbb{P}\ler{|f(x_t)-f(Y_t)|>\frac{1}{K} \text{ for inf. many }t}=1 \text{, where }Y_t\sim Q_{x_t}}\\
    &=\lers{(x_t)_t\in M^\N : \sum_t Q_{x_t}\ler{M_{x_t}}=\infty},
    \end{align*}
where for any $x\in M$,
\[
M_x:=\lers{y\in M : |f(x)-f(y)|>\frac{1}{K}}
\]
 is the subset of states at which the value of $f$ differs from $f(x)$ by more than $\frac{1}{K}$.
But
\[
\sum_t Q_{x_t}\ler{M_{x_t}}\leq c\cdot\sum_t \tilde{P}_{x_t}\ler{M_{x_t}},
\]
and so
\[
B_K=\lers{(x_t)_t\in M^\N : \sum_tQ_{x_t}(M_{x_t})=\infty}\subseteq\lers{(x_t)_t\in M^\N : \sum_t\tilde{P}_{x_t}(M_{x_t})=\infty},
\]
implying that
\[
0<\nu_m(B_K)\leq\nu_m\ler{\lers{(x_t)_t\in M^\N : \sum_t\tilde{P}_{x_t}(M_{x_t})=\infty}}.
\]

    Now the sequence $(f(x_t))_t$ is Cauchy for $\nu_m$-almost every trajectory $(x_t)_t$, and so similarly as before, there exists a positive integer $T$ such that the set
    \[
    R:=\lers{(x_t)_t\in M^\N : \sum_t\tilde{P}_{x_t}(M_{x_t})=\infty \text{ and }|f(x_i)-f(x_j)|\leq\frac{1}{K} \text{ for all }i,j\geq T}
    \]
    satisfies $\nu_m(R)>0$.
    
    Let now $R_{d,i}^\to$ be the subset of trajectories whose initial segments lie in $R$ but which take a non-Cauchy step at time $T+i$, where $d$ is the distance from the index witnessing non-Cauchiness. That is,
    \begin{align*}
    R_{d,i}^\to:=\bigg\{(w_t)_t\in M^\N : \exists(x_t)\in R \text{ such that }w_t=x_t \text{ for all }t\leq T+i\\
    \text{ and }|f(w_{T+i-d})-f(w_{T+i})|>\frac{1}{K}\bigg\}.
    \end{align*}
    Then since $R_{d,i}^\to\cap R_{d,j}^\to =\emptyset$ for any $i\neq j$ and fixed natural number $d$, we see that
\begin{align*}
    1\geq\nu_m\ler{\bigcup_{i=d}^\infty R_{d,i}^\to}
    =\sum_{i=d}^\infty \nu_m\ler{R_{d,i}^\to}&\geq\sum_{i=d}^\infty\int_R P^d_{x_{T+i-d}}\ler{M_{x_{T+i-d}}}\, d\nu_m((x_t)_t)\\
    &=\int_R \sum_{i=0}^\infty P_{x_{T+i}}^d\ler{M_{x_{T+i}}}\, d\nu_m((x_t)_t).
\end{align*}
Multiplying by $\frac{1}{2^{d+1}}$, we obtain
\[
\frac{1}{2^{d+1}}\geq\int_R\sum_{i=0}^\infty\frac{1}{2^{d+1}}P^d_{x_{T+i}}\ler{M_{x_{T+i}}} \, d\nu_m((x_t)_t),
\]
and so summing over $d$ gives
\[
1=\sum_{d=0}^\infty\frac{1}{2^{d+1}}
\geq\int_R\sum_{i=0}^\infty\sum_{d=0}^\infty\frac{1}{2^{d+1}}P^d_{x_{T+i}}\ler{M_{x_{T+i}}} \, d\nu_m((x_t))
=\int_R\sum_{i=0}^\infty\tilde{P}_{x_{T+i}}\ler{M_{x_{T+i}}} \, d\nu_m((x_t))=\infty,
\]
which is a contradiction, and so the proposition holds.
\end{proof}

\begin{proof}[Proof of Theorem \ref{thm:generalMarkov}]
    Let $h\in H^\infty(M,P)$ and $m\in M$ be fixed.
    We aim to show that $h$ is $Q$-harmonic at $m$.
    To this end, we define random variables
    \[
    M_t:=h\ler{W_t^{(m)}} \ \text{ and } \ N_t:=h\ler{Z_t^{(m)}}.
    \]
    Just as in the proof of Theorem \ref{thm:asmpt-commuting}, we will be done if we show the following three claims.
    \begin{enumerate}[label=\arabic*)]
        \item $\abs{\E\lesq{M_t}-\E\lesq{N_t}}\xrightarrow{t\to\infty}0$,
        \item $\E\lesq{N_t}\xrightarrow{t\to\infty}h(m)$, and
        \item $\E\lesq{M_t}=\int h \, dQ_m$.
    \end{enumerate}
    This is because we readily conclude from them that $h(m)=\int h \, dQ_m$, that is, $h$ is $Q$-harmonic at $m$. But $m$ was arbitrary, so $h$ is $Q$-harmonic.    

    The third claim holds simply thanks to $P$-harmonicity of $h$, as
    \[
    \E[M_t]=\E\lesq{h\ler{W_t^{(m)}}}=\int h \, d\delta_mQP^t=\int h \, dQ_m.
    \]
    The first claim follows from that
    \begin{multline*}
     \abs{\E\lesq{M_t}-\E\lesq{N_t}}=\abs{\int_M h(x) \, d\ler{\delta_mQP^t-\delta_mP^tQ}(x)} \\
    \leq \norm{h}_\infty \norm{\delta_mQP^t-\delta_mP^tQ}_{TV} \to0 \text{ as } t\to\infty.   
    \end{multline*}

    Finally, to prove the second claim, we consider
    \begin{multline}\label{expression} 
    \abs{\E[N_t]-h(m)}
    =\abs{\E\lesq{h\ler{Z_t^{(m)}}}-\E\lesq{h\ler{X_t^{(m)}}}}\\
    \leq\E\lesq{\abs{h\ler{Z_t^{(m)}}-h\ler{X_t^{(m)}}}}
    =\int_M\int_M\abs{h(x)-h(y)} \, dQ_x(y) \, dP_m^t(x),
    \end{multline}
    where we again used the $P$-harmonicity of $h$ to replace $h(m)$ with $\E\lesq{h\ler{X_t^{(m)}}}$. By the bounded martingale convergence theorem, $(h(w_n))_n$ converges for $\nu_m$-almost every trajectory $(w_n)_n$, and so by Proposition \ref{prop:LemmaForMCh}, the integrand of (\ref{expression}) goes to 0 almost surely as $t\to\infty$. Since $h$ is also bounded, we conclude that (\ref{expression}) must tend to 0 as claimed.
\end{proof}

\section{Necessary and sufficient condition}

In Section \ref{section:proof}, we derived a sufficient condition - asymptotic commutativity - for the inclusion of $H^\infty(G,\mu)$ in $H^\infty(G,\theta)$, where $\mu$ and $\theta$ are probability measures on a group $G$. This condition, however, is not necessary. In this section, we discuss if-and-only-if conditions for the inclusion to hold. 

We begin with an example that demonstrates that the asymptotic commutativity is not necessary for the inclusion of harmonic functions. Such a formal example can be easily constructed on nilpotent groups, say, but the following example is not Liouville.

\begin{example}
    Let $L(\Z)$ be the lamplighter group as in Example~\ref{example}. We construct two probability measures on $L(\Z)$ that have the same bounded harmonic functions, but neither asymptotically commutes with the other. For the sake of simplicity, the measures in this construction do not generate the group, though this is not an essential feature of the counter-example.
    
    Let $\theta$ be ``lazily going to the right and randomly switching'', namely
    \[
    \theta=\frac{1}{4}\left(\delta_{(0,1)}+\delta_{(\mathds
    {1}_{\{1\}},1)}\right)+\frac{1}{2}\delta_{(0,0)},
    \]
    and let
    \[
    \mu=\frac{1}{8}\left(\delta_{(0,1)}+\delta_{(\mathds
    {1}_{\{1\}},1)}+ \delta_{(\mathds
    {1}_{\{2\}},1)}+\delta_{(\mathds
    {1}_{\{1,2\}},1)} \right)+\frac{1}{2}\delta_{(0,0)}.
    \]

    In words, both measures go to the right in a lazy fashion, and $\theta$ switches the lamp to which it arrives with probability $1/2$ (if it was not lazy at this time), while $\mu$ does the same, but in addition, switches the next lamp to the right with probability $1/2$.

    Note that laziness is important here to avoid parity issues.
\end{example}

    It is not hard to see that the two measures in the example above do not asymptotically commute, but they do share the same bounded harmonic functions. However, one can see from their descriptions that they satisfy that both $\|\theta*\mu^{*n}-\mu^{*n}\|_{TV}$ and $\|\mu*\theta^{*n}-\theta^{*n}\|_{TV}$ tend to zero as $n\to \infty$.

    Theorem~\ref{thm:TFAE} below verifies that this is indeed a sufficient and necessary condition for lazy random walks,  as
can be derived from a classical result of Derriennic (Théorème 1 in~\cite{Derriennic}), which we restate here in English and reprove in the Appendix. Furthermore, we add another perspective, which is based on weak* topology and not total variation. To state the condition formally, we need the next section as a background. 

\subsection{$\mu$-Boundaries as hitting measures}\label{Sec:hitting}
Before stating our necessary and sufficient condition for inclusion of bounded harmonic functions, we prove a lemma that has an independent interest for stationary dynamical theory.

The first example one has in mind of a $\mu$-boundary is of a hitting measure: consider, for example, the simple random walk on the free group $\mathbb{F}_2$ with its Gromov boundary $\partial \mathbb{F}_2$. It is easy to see that almost every path of the random walk will converge, in the topology on $\mathbb{F}_2 \cup \partial\mathbb{F}_2$, to a (random) point in $\partial\mathbb{F}_2$. Then $\partial \mathbb{F}_2$ together with the distribution of the hitting point, called the hitting measure, is the canonical example of a $\mu$-boundary. 

Next, we prove that every abstract $\mu$-boundary (that is a probability space with no topology) arises in this way. 

    Let $G$ be a locally compact second countable group and $\mu$ an admissible probability measure on $G$, and let $(X,\kappa)$ be a $\mu$-boundary.
\begin{definition}
    A \emph{hitting model for an abstract $\mu$-boundary $(X,\kappa)$} is a topological space of the form $G \cup K$, such that $\mu$-almost every path in $G$ converges to a point in $K$, and the corresponding hitting measure $\nu\in Prob(K)$ satisfies that $(K,\nu)$ is $G$-equivariantly isomorphic to $(X,\kappa)$.
\end{definition}

Note that we do not require the topology on $K$ or $G\cup K$ to be compact, as one might want. The reason is that for our needs, it is redundant as it is enough that ``it looks compact as far as the random walk is concerned'' in the sense that we get convergence along $\mu$-random paths. 

\begin{lemma}\label{lemma:hitting}
    Let $G$ be a locally compact second countable group and $\mu$ an admissible probability measure on $G$.
    Then every $\mu$-boundary $(X,\kappa)$ admits a hitting model.
\end{lemma}

\begin{proof}
    Let $(X,\kappa)$ be a $\mu$-boundary, and let  $K$ be a compact metrisable model of the $\mu$-boundary $(X,\kappa)$ (as in Theorem~2.10 of~\cite{BaderShalom2006}). That is, $K$ is a compact metric space on which $G$ acts continuously, and there exists a measurable isomorphism $\phi\colon X \to K$ of $G$-spaces. Set $\nu=\phi_*\kappa\in Prob(K)$, so $(K,\nu)$ is a model of the $\mu$-boundary $(X,\kappa)$. 
    Next, use the measure $\nu\in Prob(K)$ to define a topology on $G \cup K$ using the ``measure topology'' method described in Section~3.3 of~\cite{Furst71}. Namely, a sequence $(g_n)_n$ of group elements converges to a point $k\in K$ if and only if $\ler{g_n\nu}_n$ converges to $\delta_k$ in the weak* topology on $Prob(K)$. This is indeed a well defined topology as spelled out in~\cite{compactifications}, where this construction takes the name ``orbital compactification''. Note that we do not claim that $G \cup K$ is necessarily compact.

    Since the pair $(K,\nu)$ is a $\mu$-boundary, by~\cite{BaderShalom2006}, almost every path $(w_n)_n\in G^\N$ satisfies that $w_n \nu \to \delta_k$  for some $k\in K$, where $k$ depends on the path $(w_n)_n$. By the definition of the topology on $G\cup K$, this implies that $w_n \to k$. It follows that $\nu$ is the $\mu$-hitting measure.
\end{proof}

\subsection{The weak* condition}\label{sec:weak}

\begin{proposition}\label{prop:iff}
    Let $(G,\mathcal{B}(G))$ be a topological group together with its Borel $\sigma$-algebra, and let $\mu$ and $\theta$ be probability measures on $(G,\mathcal{B}(G))$. 
    Assume that there exists a hitting model $(K,\nu)$ for the Poisson boundary of $(G,\mu)$.
       
 Then $\theta*\mu^{*n}-\mu^{*n}\rightarrow 0$ in the weak* topology on the space of measures on $G\cup K$ if and only if $H^\infty(G,\mu)\subseteq H^\infty(G,\theta)$.
\end{proposition}
\begin{proof}
    Let us denote by $X_n$ the $n$-th step of the $\mu$-random walk, and by $Y$ an independent $G$-valued random variable with distribution $\theta$. By the definition of a hitting model, we have that $\lim_nX_n\in K$ almost surely.
    Let $X_\infty:=\lim_nX_n$, 
    so we that $X_n\rightarrow X_\infty$ almost surely and $X_\infty$ has distribution $\nu$. Almost surely convergence implies that
     $X_n\rightarrow X$ in distribution, which is exactly the same as the weak* convergence of $\mu^{*n}\rightarrow\nu$.
    By continuity of the convolution operator, $\theta*\mu^{*n}\rightarrow\theta*\nu$ in the weak* topology. Hence, the sequences $\ler{\theta*\mu^{*n}}_n$ and $\ler{\mu^{*n}}_n$ weak*-converge to the same measure if and only if the $\mu$-stationary $G$-space $(K,\nu)$ is also $\theta$-stationary. Since it is a hitting model for the Poisson boundary of $(G,\mu)$, the last happens if and only if $H^\infty(G,\mu)\subseteq H^\infty(G,\theta).$
\end{proof}
For a large class of groups, we know that a hitting model for the Poisson boundary always exists, and so can get the following corollary immediately.
\begin{corollary}\label{iff, maximality}
    Let $(G,\mu)$ be a locally compact second countable group together with an admissible probability measure, and let $\theta$ be a Borel probability measure on $G$.
    Consider a hitting model $G \cup (K,\nu)$ for the Poisson boundary of $(G,\mu)$.

 Then $\theta*\mu^{*n}-\mu^{*n}\rightarrow 0$ in the weak* topology on the space of measures on $G\cup K$ if and only if $H^\infty(G,\mu)\subseteq H^\infty(G,\theta)$.
\end{corollary}

\begin{proof}
Follows from Lemma \ref{lemma:hitting} and Proposition \ref{prop:iff}.
\end{proof}

\subsection{The total-variation condition}\label{sec:TV}

In this section, we provide a translated and adapted version of a classical result by Derriennic (Théorème 1 in~\cite{Derriennic}) originally published in French.
The setup of Derriennic, which we now turn to describe, is that of operators on Banach spaces. 

Let $(E,\mathcal{F})$ be a measurable space -- for our purposes, this is usually a topological group with its Borel or Baire $\sigma$-algebra. Let $P\colon E\times\mathcal{F}\lora[0,1]$ be a Markov kernel on $E$, i.e., a transition matrix in the case when $E$ is countable. Let $\mathcal{M}$ denote the set of bounded signed measures on $(E,\mathcal{F})$ -- recall that a signed measure $\nu$ is bounded if $|\nu(F)|<\infty$ for all $F\in\mathcal{F}$, or in other words, if it has finite total variation. The tail $\sigma$-algebra $\mathcal{A}_\infty$ and the shift-invariant $\sigma$-algebra $\mathcal{S}$ are subalgebras of $\mathcal{F}^\N$, the product $\sigma$-algebra on $E^\N$, defined by
\[
\mathcal{S}:=\lers{F\in\mathcal{F}^\N : S^{-1}(F)=F}\subseteq\bigcap_{n\geq1}S^{-n}\ler{\mathcal{F}^\N}=:\mathcal{A}_\infty,
\]
where
\begin{align*}
    S\colon E^\N&\lora E^\N\\
    (x_0,x_1,x_2,\dots)&\mapsto(x_1,x_2,x_3,\dots)
\end{align*}
is the shift map.
For a sub-$\sigma$-algebra $\mathcal{G}\subseteq \mathcal{F}^\mathbb{N}$ and a measure $\kappa$ on $E$, $P_\kappa$ is the associated Markov measure on the path space $E^\mathbb{N}$, and
\[
\|P_\kappa\|_{\mathcal{G}}:=\sup\lers{\int_{E^\mathbb{N}} f\,\,dP_\kappa : f\in L^\infty(E^\mathbb{N},\mathcal{G}), \|f\|_{\infty}\leq 1}
\]
is the operator norm of $P_\kappa$ considered as a linear functional on $L^\infty\ler{E^\N,\mathcal{G}}$.

Furthermore, $H(P)$ are the $P$-harmonic functions on $E$, $D(P)\supseteq H(P)$ is the set of $P$-parabolic functions on $E$ in the sense
\begin{align*}    
D(P):= \{g\in \ell^\infty(E,\mathcal{F}) \, : \, & \exists (g_n)_{n\geq0} \text{ in } \ell^\infty(E,\mathcal{F}) \text{ with } g=g_0\\
&\text{ such that } g_n=P^*g_{n+1} \text{ for all }n\geq0\},
\end{align*}
and $D^1(P)$ is the subset of $D(P)$ defined by
\begin{align*}    
D^1(P):= \{g\in \ell^\infty(E,\mathcal{F}) \, : \, & \exists (g_n)_{n\geq0} \text{ in } \ell^\infty(E,\mathcal{F}) \text{ with } g=g_0\\
&\text{ such that } \|g_n\|\leq1 \text{ and } g_n=P^*g_{n+1} \text{ for all }n\geq0\}.
\end{align*}
\begin{theorem}[Théorème 1 in \cite{Derriennic}]\label{thm:Derriennic}
    For all signed measures $\kappa\in\mathcal{M}$ on $E$,
    \[
    \lim_{n\to\infty}\|\kappa P^n\|=\|P_\kappa\|_{\mathcal{A}_\infty}=\sup\lers{\int_E f \, d\kappa : f\in D^1(P)}
    \]
    and
    \[
    \lim_{n\to\infty}\frac{1}{n}\left\|\sum_{i=1}^n\kappa P^i\right\|=\|P_\kappa\|_{\mathcal{S}}=\sup\lers{\int_E h \, d\kappa : h\in H(P), \|h\|\leq1}.
    \]
\end{theorem}
As a corollary we get the following.
\begin{proposition}\label{iff, Cesàro}
    Let $(G,\mathcal{F})$ be a group equipped with a $G$-invariant $\sigma$-algebra $\mathcal{F}$ on $G$. Let $\theta$ and $\mu$ be two probability measures on $(G,\mathcal{F})$. Then
    \begin{equation}\label{eqn:lim=sup}
    \lim_{n\to\infty}\frac{1}{n}\left\|\sum_{i=1}^n\ler{\theta*\mu^{*i}-\mu^{*i}}\right\|_{TV}
    =\sup\lers{(\theta*h-h)(e) : h\in H(\mu), \|h\|_\infty\leq1}.
    \end{equation}
    In particular,
    \begin{equation}\label{eqn:iff}
    \frac{1}{n}\left\|\sum_{i=1}^n\ler{\theta*\mu^{*i}-\mu^{*i}}\right\|_{TV}\xrightarrow{n\rightarrow\infty} 0 \iff H^\infty(G,\mu)\subseteq H^\infty(G,\theta).
    \end{equation}
\end{proposition}
\begin{proof}
    The measure $\mu$ gives rise to a Markov kernel $P$ by setting $P_x(A):=\mu\ler{x^{-1}A}$ for any $x\in G, A\in\mathcal{F}$. To obtain (\ref{eqn:lim=sup}), we then take $\kappa=\theta-\delta_e$ in Theorem \ref{thm:Derriennic}. 
    
For (\ref{eqn:iff}), if $H^\infty(\mu)\subseteq H^\infty(\theta)$, then $\theta*h=h$ for all $h\in H^\infty(\mu)$, so in particular $(\theta*h)(e)=h(e)$ for all $h\in H^\infty(\mu)$ with $\|h\|\leq1$, and so the right-hand side of (\ref{eqn:lim=sup}) is equal to 0, giving us that
\[
\frac{1}{n}\left\|\sum_{i=1}^n\ler{\theta*\mu^{*i}-\mu^{*i}}\right\|\to0 \ \text{as }\, n\to\infty.
\]

    On the other hand, if there exists $h\in H^\infty(\mu) \setminus H^\infty(\theta)$, then there must be a $g\in G$ such that $(\theta*h)(g)\neq h(g)$. But then
    \[
    \varphi\colon x\mapsto \frac{1}{\|h\|_\infty}\cdot h(gx)
    \]
    is a well-defined $\mu$-harmonic function such that $\|\varphi\|\leq1$ and $(\theta*\varphi)(e)\neq\varphi(e)$. Since $-\varphi$ satisfies these three properties too, equality (\ref{eqn:lim=sup}) gives us that
    \[
    \lim_{n\to\infty}\frac{1}{n}\left\|\sum_{i=1}^n\ler{\theta*\mu^{*i}-\mu^{*i}}\right\|\geq|(\theta*\varphi-\varphi)(e)|>0.
    \]
\end{proof}
A harmonic function is always parabolic, hence from Proposition~\ref{iff, Cesàro}, we get that for every two probability measures $\mu$ and $\theta$ on $G$,
\[
\lim_n\frac{1}{n}\left\|\sum_{i=1}^n\ler{\theta*\mu^{*i}-\mu^{*i}}\right\|\leq\lim_n \left\|\theta*\mu^{*n}-\mu^{*n}\right\|.
\]
One can now ask under what assumptions we can conclude that 
\begin{equation}\label{iffLimits}
    \|\theta*\mu^{*n}-\mu^{*n}\|\rightarrow0 \text{ if and only if } \frac{1}{n}\left\|\sum_{i=1}^n\ler{\theta*\mu^{*i}-\mu^{*i}}\right\|\rightarrow 0.
\end{equation} 
Equivalence (\ref{iffLimits}) will hold for example if every $\mu$-parabolic function is $\mu$-harmonic.
If one knows that (\ref{iffLimits}) holds, one concludes from Proposition \ref{iff, Cesàro} that
\[
\|\theta*\mu^{*n}-\mu^{*n}\|\rightarrow0 \text{ if and only if } H^\infty(G,\mu)\subseteq H^\infty(G,\theta).
\]
\begin{corollary}
    Let $(G,\mathcal{F})$ be a group equipped with a $G$-invariant $\sigma$-algebra $\mathcal{F}$ on $G$ and $\mu$ a lazy probability measure on $(G,\mathcal{F})$. Then for every probability measure $\theta$ on $G$,
    \[
    \|\theta*\mu^{*n}-\mu^{*n}\|_{TV}\rightarrow 0 \iff H^\infty(G,\mu)\subseteq H^\infty(G,\theta).
    \]
\end{corollary}
\begin{proof}
From~\cite[Lemma 7.4.1]{Ariel}, we know that every $\mu$-parabolic function on $G$ is $\mu$-harmonic, so by the discussion above we get the result. (Although the argument in ~\cite[Lemma 7.4.1]{Ariel} is stated for finitely generated groups, it easily adapts to this generality.)
\end{proof}
In fact, since taking a lazy version of a measure does not change the corresponding space of harmonic functions, we get the following more general statement.
\begin{corollary}\label{cor:der}
    Let $(G,\mathcal{F})$ be a group equipped with a $G$-invariant $\sigma$-algebra $\mathcal{F}$ on $G$ and $\mu$ a probability measure on $(G,\mathcal{F})$. Then for every probability measure $\theta$ on $G$,
    \[
    \left\|\theta*\tilde{\mu}^{*n}-\tilde{\mu}^{*n}\right\|_{TV}\rightarrow 0 \iff H^\infty(G,\mu)\subseteq H^\infty(G,\theta),
    \]
where $\tilde{\mu}=\frac{1}{2}\mu+\frac{1}{2}\delta_e.$
\end{corollary}
\begin{remark}
   Considering the special case of $\theta=\delta_g$, we obtain the widely used criterion that a lazy $\mu$-random walk on a group $G$ is Liouville, meaning the only bounded harmonic functions it admits are constant, if and only if for every $g\in G$,
\[
\lim_{n\to\infty}\|g*\mu^{*n}-\mu^{*n}\|_{TV}=0.
\] 
\end{remark}

\subsection{Concluding the sufficient and necessary condition}

Combining the results of Sections~\ref{sec:weak} and~\ref{sec:TV}, we get the following.

\begin{theorem}\label{thm:TFAE}
    Let $G$ be a locally compact second countable group and let $\mu$ be an admissible Borel probability measure on $G$. Let $\theta$ be a Borel probability measure on $G$.
    Consider a hitting model $G \cup K$ for the Poisson boundary of $(G,\mu)$.

    Then, the following are equivalent.

    \begin{enumerate}[label=(\arabic*)]
        \item Every bounded $\mu$-harmonic function is $\theta$-harmonic: $H^{\infty}(G,\mu)\subseteq H^\infty(G,\theta).$
        \item  $\|\theta*\tilde{\mu}^{*n}-\tilde{\mu}^{*n}\|_{TV} \rightarrow 0$ where TV is the total-variation norm and $\tilde{\mu}=\frac{1}{2}(\mu+\delta_e)$.
        \item  $\theta*\mu^{*n}-\mu^{*n}\rightarrow 0$ in the weak* topology on $Prob(G\cup K)$. 
    \end{enumerate}
\end{theorem}

The reason we use $\tilde{\mu}$ in item (2) instead of $\mu$ is the potential difference between the values of the two expressions in Theorem~\ref{thm:Derriennic}.
For random walks for which the tail and the invariant $\sigma$-algebras coincide mod 0 for every initial distribution -- or equivalently, all $D^1(P)$ are $\mu$-harmonic  -- one can use the original $\mu$.

\begin{proof}
    The equivalence of (1) and (2) is the content of Corollary~\ref{cor:der}, and the equivalence of (1) and (3) is from Proposition~\ref{iff, maximality}.  
\end{proof}

\begin{remark}
It is interesting to compare conditions (2) and (3) in Theorem~\ref{thm:TFAE}. While it is standard that norm convergence implies weak* convergence, the converse implication, $(3) \implies (2)$, is analytically subtle. The weak* convergence in (3) implies that $\theta * \mu^{*n} - \mu^{*n}$ tends to zero when tested against continuous functions on the union $G \cup K$. In contrast, the total-variation convergence in (2) requires uniform control against all measurable functions bounded by 1. In general, upgrading weak* convergence to norm convergence requires additional assumptions; indeed, the direct implication here is not obvious without exploiting the connection to bounded harmonic functions as in (1).
\end{remark}
\section{Applications}\label{sec:applications}
Next, we show several examples of how some classical results follow smoothly from our main theorem.

\subsection{The centre and bounded harmonic functions}
As we already implicitly mentioned in the statement of Theorem~\ref{thm:Furstenberg}, any group acts on its own (harmonic) functions by left translation. That is,
\[
(g\cdot f)(-)=f\ler{g^{-1}\cdot-}
\]
for any function $f$ on $G$ and a group element $g$.

Scattered through the literature are many proofs, of varying levels of generality, that abelian groups have no bounded harmonic functions but constant ones (for spread-out measures -- otherwise, the group breaks down to cosets of the subgroup generated by the support, and a function is harmonic if it is constant on each coset). 
The more general formulation is that the centre $Z(G)$ acts trivially on the set of bounded harmonic functions, from which the generalisation to nilpotent groups follows.

Theorem~\ref{thm:asmpt-commuting} provides another point of view on this classical fact.

\begin{corollary}
    Let $G$ be a topological group equipped with a Borel probability measure $\mu$, and let $h$ be a bounded $\mu$-harmonic function on $G$. Then for every $g\in G$,
    \[
    \tilde{\mu}\ler{\lers{z\in Z(G) : h(zg)\neq h(g)}}=0.
    \]
\end{corollary}

\begin{proof}
    Suppose on the contrary that there exists $g\in G$ such that
    \[
    \tilde{\mu}\ler{\lers{z\in Z(G) : h(zg)\neq h(g)}}>0.
    \]
    Then we can assume without loss of generality that
    \[
    \tilde{\mu}\ler{\lers{z\in Z(G) : \text{Re }h(zg)>\text{Re }h(g)}}>0.
    \]
    Let us now consider the probability measure $\theta$ defined by
    \[
 \theta(A):=\frac{\tilde{\mu}\ler{\lers{A\cap\lers{z\in Z(G) : \text{Re }h(zg)>\text{Re }h(g)}}}}{\tilde{\mu}\ler{\lers{z\in Z(G) : \text{Re }h(zg)>\text{Re }h(g)}}}
    \]
    for any measurable $A\subseteq G$.
    Then $\theta\ll\tilde{\mu}$ and $\theta*\mu=\mu*\theta$, hence by Theorem~\ref{thm:asmpt-commuting}, the function $h$ is $\theta$-harmonic, so in particular
    \[
    \text{Re }h(g)=\int_G \text{Re }h(gy) \, d\theta(y)=\int_G \text{Re }h(yg) \, d\theta(y)>\text{Re }h(g),
    \]
    which is a contradiction.
\end{proof}

If the bounded function $h$ is not just harmonic, but also continuous (this is necessarily the case whenever $G$ is locally compact and second countable and $\Tilde{\mu}$ is not singular with respect to the Haar measure, as Babillot points out in Lemma 1.2 in~\cite{proceedings}) and $G$ is second countable, then the corollary can be strengthened to say that
\[
\tilde{\mu}\ler{\lers{z\in Z(G) : z^{-1}\cdot h\neq h}}=0.
\]
Indeed,
\[
\lers{z\in Z(G) : z^{-1}\cdot h\neq h} = Z(G)\cap\bigcup_{g\in G}\lers{\gamma\in G : |h(\gamma g)-h(g)|>0},
\]
and for every $g\in G$, the open set $\lers{\gamma\in G : |h(\gamma g)-h(g)|>0}$ is a union $\bigcup_{n=1}^\infty B_{g,n}$ of basic open sets.
The union $\bigcup_{g\in G}\lers{\gamma\in G : |h(\gamma g)-h(g)|>0}$ can thus be written as $\bigcup_{g\in G}\bigcup_{n=1}^\infty B_{g,n}$, but there are only countably many open basic sets, so there must be $g_i\in G$ and $n_i\in\N$ such that
\[
\bigcup_{g\in G}\bigcup_{n=1}^\infty B_{g,n}=\bigcup_{i=1}^\infty B_{g_i,n_i}.
\]
Therefore,
\[
\tilde{\mu}\ler{\lers{z\in Z(G)  : z^{-1}\cdot h\neq h}}=\sum_{i=1}^\infty\tilde{\mu}\ler{Z(G)\cap B_{g_i,n_i}},
\]
so $\tilde{\mu}\ler{\lers{z\in Z(G) : z^{-1}\cdot h\neq h}}>0$ would imply the existence of $i\in\N$ such that $\tilde{\mu}\ler{Z(G)\cap B_{g_i,n_i}}>0$ and we would proceed as before, contradicting harmonicity at $g_i$.

Similarly, in Theorem~\ref{thm:asmpt-commuting}, knowing that all bounded $\mu$-harmonic functions are continuous would allow us to replace the assumption $\norm{\theta*\mu^{*t}-\mu^{*t}*\theta}_{TV}\xrightarrow{t\rightarrow\infty}0$ with the weaker one saying that $\theta*\mu^{*t}-\mu^{*t}*\theta$ goes to the zero measure in the weak* topology.

\subsection{Convex combinations of convolution powers}

It is often helpful to be able to assume that a probability measure with respect to which we perform a random walk is fully supported. The following classical result says we can always assume this while keeping the set of bounded harmonic functions intact.
It can be proved by separating a $(\alpha_1\mu+\alpha_2\mu^{*2}+...)$-random walk to two parts, where one is a random walk on $\N$, as in Proposition 2.4 in
\cite{hartman2015furstenberg} (or by other means~\cite{forghani2019positive}, where \textit{positive} harmonic functions are discussed). Theorem \ref{thm:asmpt-commuting} provides a one-line proof, which moreover does not place any restrictions on the group $G$.

\begin{proposition}
     Let $G$ be a group equipped with a $G$-invariant $\sigma$-algebra $\Sigma$ and a probability measure $\mu$ on $(G,\Sigma)$. Let $(\alpha_n)_{n=0}^\infty$ be a sequence of real non-negative numbers such that $\sum_{n=0}^\infty\alpha_n=1$ and $\alpha_1>0$. We define a probability measure
     \[
     \theta:=\sum_{n=0}^\infty\alpha_n\mu^{*n}.
     \] 
     Then $H^\infty(G,\mu)= H^\infty(G,\theta)$.
\end{proposition}
\begin{proof}
    As $\theta$ is a convex combination of powers of $\mu$, we have $\mu*\theta=\theta*\mu$.
    Moreover,
    \[
    \theta\ll \sum
    \frac{1}{2^{n+1}}\mu^{*n},
    \]
    and since $\alpha_1>0$,
    \[
    \mu\ll\sum\frac{1}{2^{n+1}}\theta^{*n}.
    \]
    Then by Theorem \ref{thm:asmpt-commuting}, we get that $H^\infty(G,\mu)= H^\infty(G,\theta)$.
\end{proof}
\begin{corollary}
     Let $(G,\Sigma)$ and a probability measure $\mu$ be as above. Let $t\in(0,1)$ be fixed. Define the probability measure $\theta=t\mu+(1-t)\delta_e$, where $\delta_e$ is the Dirac measure on the identity element.
    Then $H^\infty(G,\mu)=H^\infty(G,\theta).$ In other words, when talking about harmonic functions, one can always assume that the measure is \textit{lazy}, that is, its support includes the identity element as an atom.
\end{corollary}

\subsection{Restricting bounded harmonic functions on $G_1\times G_2$}
In this subsection, we consider the question of the restriction of a bounded harmonic function on a product group to one of the factors that is still harmonic. Namely, let $\mu_i$ be a probability measure on $G_i$ for two groups $G_1$ and $G_2$, and consider $G=G_1\times G_2$ and $\mu=\mu_1 \times \mu_2$.

The setup is similar to the one considered in Bader-Shalom~\cite{BaderShalom2006}, in their Factor and Intermediate Factor Theorems. In particular, one can find some similarities between our corollary below and Lemma 3.1 in~\cite{BaderShalom2006}. However, in the framework of bounded harmonic functions, one has to be a bit more careful, as we explain now.

Observe that in the greatest generality, the statement is false. Namely, the restriction is not necessarily harmonic, as the following example shows.

\begin{example}\label{Example:unrelated measures}
Let $G_1$ and $G_2$ be two copies of the free group $\mathbb{F}_2=\langle a,b \rangle$, and let both $\mu_1$ and $\mu_2$ be the simple random walk $\frac{1}{4}\ler{\delta_a+\delta_{a^{-1}}+\delta_b+\delta_{b^{-1}}}$.
Note that the product measure $\mu= \mu_1 \times \mu_2$ on $G= \mathbb{F}_2 \times \mathbb{F}_2$ is not generating, because any element that can be reached through $\mu$ will always have equal parity of word lengths in both coordinates.
Indeed, the $\mu$-random walk would keep the parity in each coordinate, coordinated.
That is, denote by $H\le G$ be the subgroup of elements $(g_1,g_2)$ such that the number $|g_1|+|g_2|$ is even. Then $supp(\mu)\subset H$, and hence the characteristic function $1_H$ on $G$ is $\mu$-harmonic. However, its restriction to either coordinate is not harmonic.
\end{example}

However, as a corollary of Theorem~\ref{thm:asmpt-commuting}, we get the following.

\begin{corollary}\label{Cor:product-group}
    Let $G_1$, $G_2$ be groups, each equipped with an invariant $\sigma$-algebra, and let $\mu_1, \mu_2$ be two probability measures on them. Suppose further that $\mu_1\ler{\lers{e_{G_1}}}>0$. Then if $h$ is a bounded $\mu_1\times\mu_2$-harmonic function on $G_1\times G_2$, the function
    \begin{align*}
        h\ler{g_1,-}\colon G_2&\lora\C\\
        \gamma&\mapsto h(g_1,\gamma)
    \end{align*}
    is $\mu_2$-harmonic for every $g_1\in G_1$.
\end{corollary}

\begin{proof}
    Let $g_1\in G_1$ be fixed.
    We have $\delta_e\times\mu_2\ll\mu_1\times\mu_2$ and $\ler{\delta_e\times\mu_2}*\ler{\mu_1\times\mu_2}=\mu_1\times\mu_2^2=\ler{\mu_1\times\mu_2}*\ler{\delta_e\times\mu_2}$, so by Theorem~\ref{thm:asmpt-commuting}, the function $h$ is $\delta_e\times\mu_2$-harmonic. This in particular means that for every $\gamma\in G_2$,
    \[
    h\ler{g_1,\gamma}=\int_{G_1\times G_2} h\ler{(g_1,\gamma)(a,b)} \, d\ler{\delta_e\times\mu_2}(a,b)
    =\int_{G_2} h\ler{g_1,\gamma b} \, d\mu_2(b),
    \]
    that is, the function $h\ler{g_1,-}$ on $G_2$ is $\mu_2$-harmonic.
\end{proof}

Note that Theorem~\ref{thm:asmpt-commuting} allows a more general setup, of almost commuting measures. However, we chose to state Corollary~\ref{Cor:product-group} for commuting measures for simplicity.

\newpage

\begin{appendices}

\appendixpage

We will first prove a general result about contractions of Banach spaces, which we will subsequently apply to the special case in which the Banach space is that of signed measures with finite total variation, obtaining Theorem~\ref{thm:Derriennic}.

\begin{thm}[Théorème 2 in~\cite{Derriennic}]
    Let $B$ be a Banach space, $B^*$ its dual, $S^*$ the closed unit ball of $B^*$, and $T\colon B\lora B$ a linear contraction. Then for every $x\in B$,
    \[
    \lim_{n\to\infty}\|T^nx\| = \sup\lers{|\langle x,x^*\rangle| \, : \, x^*\in\bigcap_{n\geq1}T^{n*}S^*}
    \]
    and
    \[
    \lim_{n\to\infty}\frac{1}{n}\left\|\sum_{i=1}^nT^ix\right\| = \sup\lers{|\langle x,x^*\rangle| \, : \, x^*\in S^*\cap I},
    \]
    where
    \[
    I=\lers{x^*\in B^* \, : \, x^*=T^*x^*}.
    \]
\end{thm}

\begin{proof}
    Let $x\in B$ be fixed. The sequence $\ler{\|T^nx\|}_{n=1}^\infty$ of non-negative real numbers is decreasing and hence convergent.
    For every positive integer $m$, we have
    \[
    \|T^mx\|=\sup_{x^*\in S^*}|\langle T^mx,x^*\rangle|
    =\sup_{x^*\in T^{*m}S^*}|\langle x,x^*\rangle|
    \geq\sup\lers{|\langle x,x^*\rangle| \, : \, x^*\in\bigcap_{n\geq1} T^{*n}S^* },
    \]
    and hence
    \[
    \lim_{n\to\infty}\|T^nx\| \geq \sup\lers{|\langle x,x^*\rangle| \, : \, x^*\in\bigcap_{n\geq1}T^{n*}S^*}.
    \]

    To show the opposite inequality, suppose that $\varepsilon>0$ is given. For every positive integer $m$, there exists $x^*_m\in S^*$ such that
    \[
    \|T^mx\|-\varepsilon\leq|\langle T^mx,x^*_m\rangle|=|\langle x,T^{*m}x^*_m\rangle|.
    \]
    Let $y^*\in S^*$ be an accumulation point of the sequence $\ler{T^{*n}x^*_n}_{n=1}^\infty$ with respect to the weak* topology on $B^*$, in which $S^*$ is compact by Alaoglu's theorem. $T^*$ is continuous, so $T^{*m}S^*$ is compact for every $m$.
    Since $T^{*(m+1)}S^*\subseteq T^{*m}S^*$, we therefore have that $y^*\in\bigcap_{n\geq1}T^{*n}S^*$.
    This means that there is a subsequence $(n_i)_{i=1}^\infty$ such that
    \[
    |\langle x,y^*\rangle|
    =\lim_{i\to\infty}\left|\langle x,T^{*n_i}x^*_{n_i}\rangle\right|\geq\lim_{n\to\infty}\|T^nx\|-\varepsilon,
    \]
    which concludes the proof of the first part of the theorem.

    To prove the second part of the theorem, observe that since $T$ is a linear contraction,
    \[
    \left\|\sum_{i=1}^{n+k}T^ix\right\|
    \leq\left\|\sum_{i=1}^{n}T^ix\right\|+\left\|\sum_{i=1}^{k}T^ix\right\|
    \]
    holds for all $n$ and $k$, implying that the sequence $\ler{\frac{1}{n}\left\|\sum_{i=1}^{n}T^ix\right\|}_{n=1}^\infty$ converges.
    If $x^*\in S^*\cap I$, then for every $m$,
    \[
    |\langle x,x^*\rangle|=\left|\left\langle \frac{1}{m}\sum_{i=1}^m T^i x,x^*\right\rangle\right|
    \leq\frac{1}{m}\left\|\sum_{i=1}^mT^ix\right\|,
    \]
    and so
    \[
    \lim_{n\to\infty}\frac{1}{n}\left\|\sum_{i=1}^nT^ix\right\|\geq\sup\lers{|\langle x,x^*\rangle| \, : \, x^*\in S^*\cap I}.
    \]

    For the other direction, let $\varepsilon>0$ be fixed.
    For every $m$, there is $x^*_m\in S^*$ such that
    \[
    \frac{1}{m}\left\|\sum_{i=1}^mT^ix\right\|-\varepsilon\leq\left|\left\langle x,\frac{1}{m}\sum_{i=1}^mT^{*i}x^*_m\right\rangle\right|.
    \]
    Let $z^*\in S^*$ be an accumulation point of the sequence $\ler{\frac{1}{n}\sum_{i=1}^nT^{*i}x^*_n}_{n=1}^\infty$ with respect to the weak* topology on $B^*$.
    For every $y\in B$, the number $\left|\langle y-Ty,z^*\rangle\right|$ is an accumulation point of the sequence $\ler{\left|\left\langle y-Ty,\frac{1}{n}\sum_{i=1}^n T^{*i}x^*_n\right\rangle\right|}_{n=1}^\infty$,
    but
    \[
    \left|\left\langle y-Ty,\frac{1}{m}\sum_{i=1}^m T^{*i}x^*_m\right\rangle\right|
    =\left|\left\langle y,\frac{1}{m} T^{*}x^*_m\right\rangle-\left\langle y,\frac{1}{m} T^{*(m+1)}x^*_m\right\rangle\right|
    \leq\frac{2}{m}\|y\|
    \]
    for every positive integer $m$, and so we must have
    $0=\left|\langle y-Ty,z^*\rangle\right|=\left|\langle y,z^*\rangle-\langle y,T^*z^*\rangle\right|$
    and therefore $z^*\in I$.
    There is therefore a subsequence $(n_j)_{j=1}^\infty$ such that
    \[
    |\langle x,z^*\rangle|=\lim_{j\to\infty}\left|\left\langle x,\frac{1}{n_j}\sum_{i=1}^{n_j}T^ix^*_{n_j}\right \rangle\right|\geq\lim_{n\to\infty}\frac{1}{n}\left\|\sum_{i=1}^nT^ix\right\|-\varepsilon,
    \]
    concluding the result.
\end{proof}

Let now $(E,\mathcal{F})$ be a measurable space, and $P\colon E\times\mathcal{F}\lora[0,1]$ a Markov kernel on $(E,\mathcal{F})$. Abusing notation, $P$ is also a contraction on $\mathcal{M}$, the Banach space of signed measures on $(E,\mathcal{F})$ of finite total variation, sending $\mu$ to $\mu P$ given by
\[
(\mu P) (A) = \int_E P(x,A) \, d\mu(x).
\]

The dual $\mathcal{M}^*$ of $\mathcal{M}$ contains the set $\ell^\infty(E,\mathcal{F})$ of bounded $\mathcal{F}$-measurable functions on $E$ equipped with the supremum norm, and the dual contraction $P^*$ on $\ell^\infty(E,\mathcal{F})$ is the averaging operator
\[
P^*\colon f \mapsto \ler{x \mapsto \int_Ef(y) \, dP_x(y)}.
\]
The theorem above thus immediately tells us that
    for all signed measures $\kappa\in\mathcal{M}$ on $E$,
    \[
    \lim_n\|\kappa P^n\|\geq\sup\lers{\int_E f \, d\kappa : f\in D^1(P)}
    \]
    and
    \[
    \lim_n\frac{1}{n}\left\|\sum_{i=1}^n\kappa P^n\right\|\geq\sup\lers{\int_E h \, d\kappa : h\in H(P), \|h\|_\infty\leq1},
    \]
    but the inequality can be strengthened to equality if we choose $B$ and $B^*$ less naively.

\begin{proof}[Proof of Theorem \ref{thm:Derriennic}]
Let a signed measure $\kappa\in\mathcal{M}$ on $E$ be fixed. Let $m$ be a non-negative $\sigma$-finite measure on $E$ such that $\kappa\ll m$ and $mP\ll m$. Such $m$ always exists -- one can take for example $m=\sum_{n=0}^\infty\ler{\kappa^+P^n+\kappa^-P^n}$, where $\kappa=\kappa^+-\kappa^-$ is the Hahn decomposition of $\kappa$. Then we can apply the theorem above to $B=L^1(m)$, $B^*=L^\infty(m)$, with the contraction $T$ on $B$ being
\begin{align*}    
T\colon L^1(m) & \to L^1(m)\\
\frac{d\mu}{dm}&\mapsto \frac{d\mu P}{dm},
\end{align*}
where we used the fact that every function in $L^1(m)$ is the Radon-Nikodym derivative of some signed measure which is absolutely continuous with respect to $m$. The map $T$ is well-defined thanks to that $mP\ll m$.
The dual of $T$ is
\begin{align*}
T^*\colon L^\infty(m)&\to L^\infty(m) \\
f &\mapsto \ler{x\mapsto\int_E f(y) \, dP_x(y)}.
\end{align*}
Specifically, considering $\frac{d\kappa}{dm}\in L^1(m)$, the theorem tells us that
\begin{align*}
    \lim_{n\to\infty}\left\|\frac{d\kappa P^n}{dm}\right\|_1=\lim_{n\to\infty}\left\|T^n\frac{d\kappa}{dm}\right\|_1&=\sup\lers{\left|\left\langle\frac{d\kappa}{dm},g\right\rangle\right| : g\in\bigcap_{n\geq1}T^{*n}L^\infty_{[-1,1]}(m)}\\
    &=\sup\lers{\int_E g \, d\kappa : g\in\bigcap_{n\geq1}T^{*n}L^\infty_{[-1,1]}(m)}
\end{align*}
and
\begin{align*}
    \lim_n\frac{1}{n}\left\|\sum_{i=1}^n\frac{d\kappa P^i}{dm}\right\|_1=\lim_n\frac{1}{n}\left\|\sum_{i=1}^nT^i\frac{d\kappa}{dm}\right\|_1&=\sup\lers{\left|\left\langle\frac{d\kappa}{dm},g\right\rangle\right| : \|g\|_\infty\leq1, T^{*}g=g\in L^\infty(m)}\\
    &=\sup\lers{\int_E g \, d\kappa : Pg=g\in L^\infty_{[-1,1]}(m)}.
\end{align*}

But by Proposition 2 in~\cite{Derriennic}, the suprema, which we are taking over $H_m$ and $D_m^1$ above are equal to those over $H$ and $D^1$, so since
$\|\kappa P^i\|_{TV}=\left\|\frac{d\kappa P^i}{dm}\right\|_1$ and
$\left\|\sum_{i=1}^n\kappa P^i\right\|_{TV}=\left\|\frac{d\sum_{i=1}^n\kappa P^i}{dm}\right\|_1=\left\|\sum_{i=1}^n\frac{d\kappa P^i}{dm}\right\|_1$ for every $n\geq1$, we are done.
\end{proof}

\end{appendices}

\newpage

\printbibliography

\end{document}